\documentclass[11pt,letterpaper,twoside]{article}

  \usepackage[margin=3cm]{geometry}

  \usepackage{amsmath,amssymb,amsthm,amsfonts,bbding,stmaryrd,dsfont,wasysym,pifont,xspace,xcolor}
  \usepackage{parskip,fancyhdr,url}

  \usepackage{fourier}

  \usepackage{tikz-cd}
  \usepackage{epstopdf,yfonts}
  \usepackage{graphicx,enumerate,epsfig} 
  \usepackage{mathrsfs}
  \usepackage{enumitem}
  \usepackage[normalem]{ulem}

  \usepackage[
     breaklinks,colorlinks,
     citecolor=black,linkcolor=black,urlcolor=black
  ]{hyperref}
  
  \usepackage{tikz}
  \usetikzlibrary{arrows,calc,decorations.pathmorphing,backgrounds,positioning,fit}

  \usepackage[all]{xy}

  \begingroup
    \makeatletter
    \@for\theoremstyle:=definition,remark,plain\do{%
      \expandafter\g@addto@macro\csname th@\theoremstyle\endcsname{%
        \addtolength\thm@preskip\parskip
        }%
      }
  \endgroup

  \newcommand\address[1]{}
  \newcommand\email[1]{}
  \newcommand\dedicatory[1]{}

  \theoremstyle{plain}
  \newtheorem{theorem}{Theorem}[section]
  \newtheorem{proposition}[theorem]{Proposition}
  \newtheorem{corollary}[theorem]{Corollary}
  \newtheorem{lemma}[theorem]{Lemma}

  \newtheorem{introthm}{Theorem}

  \theoremstyle{definition}
  \newtheorem{definition}[theorem]{Definition}
  \newtheorem{remark}[theorem]{Remark}

  \newtheorem*{claim*}{Claim}

  \newtheorem*{question*}{Questions}
  \newtheorem*{answer*}{Answer}
  \newtheorem*{application*}{Application}

  \renewcommand{\eqref}[1]{Equation~(\ref{Eq:#1})}

  \newcommand{\N}{\ensuremath{\mathbb{N}}\xspace}
  \newcommand{\Z}{\ensuremath{\mathbb{Z}}\xspace}
  \newcommand{\R}{\ensuremath{\mathbb{R}}\xspace}
  
  \newcommand{\Q}{\ensuremath{\mathbb{Q}}\xspace}

  \DeclareMathOperator{\Homeo}{Homeo}
  \DeclareMathOperator{\id}{id}
  
  \DeclareMathOperator{\image}{im}
  
  \DeclareMathOperator{\mcg}{MCG}
  \DeclareMathOperator{\pmcg}{PMCG}
  \DeclareMathOperator{\HH}{H}
  \DeclareMathOperator{\Sym}{Sym}

  \newcommand{\set}[1]{\ensuremath{\left\{ {#1} \right\}}\xspace} 

  \newcommand{\param}{{\mathchoice{\mkern1mu\mbox{\raise2.2pt\hbox{$
  \centerdot$}}
  \mkern1mu}{\mkern1mu\mbox{\raise2.2pt\hbox{$\centerdot$}}\mkern1mu}{
  \mkern1.5mu\centerdot\mkern1.5mu}{\mkern1.5mu\centerdot\mkern1.5mu}}}

  \fancypagestyle{plain}

\begin{document}


  \title    {Self-Similar Surfaces: Involutions and Perfection}
  \author   {\stepcounter{footnote}Justin Malestein\thanks{Partially supported by Simons
  Collaboration Grant 713006}\,\, and\, Jing Tao\thanks{Partially
  supported by NSF DMS-1651963.}}
  \date{}

  \maketitle
  \thispagestyle{empty}

  \begin{abstract} 
    We investigate the problem of when big mapping class groups are
    generated by involutions. Restricting our attention to the class of
    \emph{self-similar} surfaces, which are surfaces with self-similar ends
    space, as defined by Mann and Rafi, and with $0$ or infinite genus, we
    show that, when the set of maximal ends is infinite, then the mapping
    class groups of these surfaces are generated by involutions, normally
    generated by a single involution, and uniformly perfect. In fact, we
    derive this statement as a corollary of the corresponding statement for
    the homeomorphism groups of these surfaces. On the other hand, among
    self-similar surfaces with one maximal end, we produce infinitely many
    examples in which their big mapping class groups are neither perfect
    nor generated by torsion elements. These groups also do not have the
    automatic continuity property.
  \end{abstract}

\section{Introduction}

  Consider a connected and oriented surface $\Sigma$. We distinguish two
  types of surfaces, those of finite type, i.e.\ a closed surface minus
  finitely many points, or of infinite type otherwise. Let $G(\Sigma)$ be
  either the group $\Homeo^+(\Sigma)$ of orientation preserving
  self-homeomorphisms of $\Sigma$ or the mapping class group $\mcg(\Sigma)$
  of $\Sigma$. We are interested in the algebraic structure of $G(\Sigma)$,
  especially when $\Sigma$ has infinite type. 

  As a topological group, equipped with the compact open topology,
  $\Homeo^+(\Sigma)$ is a non-locally-compact Polish group. $\mcg(\Sigma)$,
  being a quotient of $\Homeo^+(\Sigma)$, inherits a topology. When
  $\Sigma$ has finite type, then this topology is discrete and
  $\mcg(\Sigma)$ is finitely presented. But when $\Sigma$ has infinite
  type, then $\mcg(\Sigma)$ is also a non-locally-compact Polish group,
  similar to the homeomorphism group. In particular, $\mcg(\Sigma)$ is not
  countably generated, justifying the nomenclature of \emph{big mapping
  class group} in the literature. 
  
  An obvious group-theoretic problem is to identify canonical generating
  sets for $G(\Sigma)$. For any group, a natural choice is its set of
  involutions, or more broadly, its set of torsion elements. This leads us
  to ask if $G(\Sigma)$ is generated by involutions (or torsion elements).
  (The set of Dehn twists, being countable, can never generate a big
  mapping class group; and often, they do not even topologically generate
  \cite{APV20}.) For finite type surfaces, this question is well studied
  for their mapping class groups; see \cite{MP87, Luo00, BF04, Kas03,
  Kor20, LM21, Yil20} and the references within for the story on generating
  by involutions. The goal of this paper is to explore this question for
  surfaces of infinite type. 
  
  To answer this question for all surfaces of infinite type should be
  challenging, as $G(\Sigma)$ is as complicated as the homeomorphism group
  of the \emph{ends space} of $\Sigma$. In trying to tame the world of
  surfaces of infinite type, Mann and Rafi \cite{MR21} introduced a
  preorder on an ends space, and showed that the induced partial order
  always has maximal elements. They also introduced the notion of
  self-similar ends spaces. We call a surface self-similar if it has a
  self-similar ends space and $0$ or infinite genus. Among these, we
  identity a subclass, called \emph{uniformly} self-similar, which are
  self-similar with infinitely many maximal ends. This subclass, which is
  uncountable, exhibits additional symmetry, to which the sphere minus a
  Cantor set belongs. It was observed by Calegari \cite{Cal09} that the
  mapping class group of the sphere minus a Cantor set is uniformly
  perfect. We extend this result to all uniformly self-similar surfaces,
  along with answering the generation problem by involutions for these
  surfaces. Our main theorem is the following.
  
  \begin{introthm} \label{introthm1}
   
    Let $\Sigma$ be a uniformly self-similar surface and $G(\Sigma)$ be
    either $\Homeo^+(\Sigma)$ or $\mcg(\Sigma)$. Then $G(\Sigma)$ is
    generated by involutions, normally generated by a single involution,
    and uniformly perfect. Moreover, each element of $G(\Sigma)$ is a
    product of at most 3 commutators and 12 involutions. 
    
  \end{introthm}
  
  Note that the case of $\mcg(\Sigma)$ follows immediately from that of
  $\Homeo^+(\Sigma)$. For the latter case, we use a method akin to
  \emph{fragmentation}, a well known tool in the study of homeomorphism
  groups. In more detail, we first observe that a uniformly self-similar
  ends space $E$ behaves very much like a Cantor set. Namely, any clopen
  subset $U \subset E$ containing a proper subset of the maximal ends is
  homeomorphic to its complement $U^c$. This gives rise to the notion of a
  \emph{half-space} in a uniformly self-similar surface $\Sigma$, which is
  a subsurface $H \subset \Sigma$ with a single connected, compact boundary
  component, such that $\overline{H^c}$ is also a half-space and
  homeomorphic to $H$. (The exact definition is different and appears as
  Definition \ref{def:halfspaces}). We then find an
  $H$--\emph{translation}, that is, a homeomorphism $\phi$ such that
  $\set{\phi^n(H)}_{n \in \Z}$ are all disjoint. This is a key step in the
  proof and requires putting the surface $\Sigma$ into a particular form
  that reflects its symmetry. By our construction, the $H$--translation
  $\phi$ is a product of two conjugate involutions. Then, using a standard
  trick, we write every $f \in \Homeo(H,\partial H)$ as a commutator of the
  form $f=[\hat{f},\phi]$ for some $\hat{f}$. The final step is  to show
  $\Homeo^+(\Sigma)$ is the normal closure of $\Homeo(H,\partial H)$, and
  so it is normally generated by $\phi$ and hence by a single involution.
  The other statements are achieved by keeping track of the number of
  commutators or involutions needed at each step. 
  
  Many of our steps above carry over to the case of equipping $\Sigma$ with
  a marked point $\ast$. The key difference is now one can find a curve
  $\alpha \subset \Sigma$ which is not contained in any half-space $H$ of
  $\Sigma$. Thus, it is no longer immediate that the Dehn twist about
  $\alpha$ can be generated by elements supported on $H$ and their
  conjugates. To deal with this issue, we invoke the lantern relation.
  Using self-similarity, we can find an appropriate lantern, i.e.\ a
  4-holed sphere bounding $\alpha$ with the other boundary components lying
  in half-spaces. Once we get all Dehn twists, then combining our previous
  method together with the fact that mapping class groups of compact
  surfaces are generated by Dehn twists, we obtain the following theorem.
  
  \begin{introthm} \label{introthm2}

    Let $\Sigma$ be a uniformly self-similar surface with a marked point
    $\ast \in \Sigma$. Then, $\mcg(\Sigma, \ast)$ is perfect, generated by
    involutions, and normally generated by a single involution. 

  \end{introthm}
  
  Because we used the lantern relation, our proof does not apply to the
  homeomorphism group. For a different argument that the mapping class
  group of the marked sphere minus a Cantor set is perfect, see
  \cite{Vla21}. Theorem \ref{introthm2} is sharp in the sense that we
  cannot expect a statement about uniform perfection or a bound on the
  number of involutions, due to the fact that the marked sphere minus a
  Cantor set provides a counterexample, by \cite{Bav16}.
  
  It is not possible for  all big mapping class groups to be generated by
  torsion elements or be perfect, even among the class of self-similar
  surfaces. One counterexample is the infinite genus surface with one end.
  This is a self-similar surface, but, by Domat and Dickmann \cite{DD21},
  the abelianization of its mapping class group contains
  $\bigoplus_{2^{\aleph_0}} \Q$ as a summand. 
  
  Using their results and a covering trick, we can show the mapping class
  group of the surface $\R^2 \setminus \N$ has similarly large
  abelianization. Note that this surface is also self-similar but not
  uniformly. On the other hand, using a method similar to our proof of
  Theorem \ref{introthm1}, we can also show $\mcg(\R^2 \setminus \N)$ is
  \emph{topologically} generated by involutions. Since any homomorphism
  from a Polish group to $\Z$ is always continuous, this makes its first
  cohomology group vanish, in contrast with homology.   

  \begin{introthm} \label{introthm3}

    The group $\mcg(\R^2 \setminus \N)$ surjects onto
    $\bigoplus_{2^{\aleph_0}} \Q$. In particular, it is not perfect or
    generated by torsion elements. On the other hand, $\mcg(\R^2 \setminus
    \N)$ is topologically generated by involutions, so $\HH^1(\mcg(\R
    \setminus \N),\Z)=0$. 

  \end{introthm}
  
  The statement of topological generation by involutions also extends to
  the mapping class group of the one-ended infinite genus surface
  $\Sigma_L$. Additionally, we can get infinitely many examples of surfaces
  whose mapping class groups have similarly large abelianization, by
  considering appropriate maps to $\Sigma_L$ or $\R^2 \setminus \N$. Many
  of these examples are self-similar but not uniformly. 
  
  Another application of our result beyond the ones mentioned is the
  \emph{automatic continuity property}. Recall a Polish group $G$ has the
  automatic continuity property if every homomorphism from $G$ to a
  separable topological group is necessarily continuous. The family of
  surfaces we construct also gives rise to a large family of homeomorphism
  groups or big mapping class groups that do not have this property. This
  gives some progress towards answering \cite[Question 2.4]{Man20}. We
  highlight the following examples and refer to Theorem \ref{thm:bigabel}
  and Corollary \ref{cor:bigabel} for the full technical statement.

  \begin{introthm} \label{introthm4}

    Let $\Sigma = S^2 \setminus E$, where $S^2$ is the 2-sphere and $E$ is
    a countable closed subset of the Cantor set homeomorphic to the ordinal
    $\omega^\alpha +1$, where $\alpha$ is a countable successor ordinal.
    Let $G(\Sigma)$ be either $\Homeo^+(\Sigma)$ or $\mcg(\Sigma)$. Then
    $G(\Sigma)$ is not perfect, is not generated by torsion elements, and
    does not have the automatic continuity property. 

  \end{introthm}
  
  One may wonder what happens in the case of positive genus, rather than
  $0$ or infinite genus. Our methods do not extend to these surfaces.
  However, for a surface $\Sigma$ obtained by removing a Cantor set from a
  surface of finite type, Calegari and Chen \cite{CC21} showed various
  results for $\mcg(\Sigma)$ including that it is generated by torsion.
  Additionally, Mann \cite{Man20} showed $G(\Sigma)$ has the automatic
  continuity property. It would be interesting to know if their techniques
  extend to uniformly self-similar ends spaces. We refer to their papers
  for more details.
  
  One may also wonder whether our results extend to other surfaces of
  infinite type. In \cite{Fie21}, using very similar methods that were
  developed independently and concurrently, Field, Patel, and Rasmussen
  proved analogues of some of the above results for other classes of
  surfaces. Specifically, for their class of surfaces, which are required
  to have locally CB mapping class group and infinitely many maximal ends
  among other minor conditions, they show that the commutator lengths of
  elements in the commutator subgroup are uniformly bounded above and
  $\HH_1(\mcg(\Sigma), \Z)$ is finitely generated. See \cite{Fie21} for
  precise statements.

  As many cases still remain open, we invite the reader to explore other
  classes of surfaces of infinite type which may verify the properties in
  Theorem \ref{introthm1} or admit an obstruction. It would also be
  interesting to find other natural generating sets for big mapping class
  groups or homeomorphism groups. Similar questions can also be asked for
  the homeomorphism groups of ends spaces. 

  Here is a brief outline of the paper. In Section \ref{sec:preliminaries},
  we introduce ends spaces and the classification of surfaces of infinite
  type. Following \cite{MR21}, we define self-similar ends spaces and
  surfaces and a partial order on ends spaces. We also observe some nice
  properties about self-similar ends spaces that lead to the definition of
  half-spaces in uniformly self-similar surfaces. The proof of Theorem
  \ref{introthm1} is contained in Section \ref{sec:genunmarked}, and the
  proof of Theorem \ref{introthm2} in Section \ref{sec:genmarked}. The two
  parts of Theorem \ref{introthm3} appear in Section \ref{sec:bigabel} as
  Proposition \ref{prop:flute} and Theorem \ref{thm:topgen}. Theorem
  \ref{introthm4} follows from Corollary \ref{cor:bigabel} as a special
  case.
 
 \paragraph{Acknowledgements} We would like to thank the anonymous referee for their helpful comments.

\section{Preliminaries}

  \label{sec:preliminaries}

  \subsection{Partial order on ends spaces}

  \label{sec:po}

  An \emph{ends space} is a pair $(E,F)$, where $E$ is a totally
  disconnected, compact, metrizable space and $F \subset E$ is a (possibly
  empty) closed subspace. For simplicity, we will often suppress the
  notation $F$, but by convention, all homeomorphisms of $E$ will be
  relative to $F$. For instance, to say $C\subset E$ is homeomorphic to $D
  \subset E$ means there is a homeomorphism from $(C,C\cap F)$ to $(D,D
  \cap F)$. We denote by $\Homeo(E,F)$ the group of homeomorphisms of $E$
  preserving $F$.
  
  The assumptions on $E$ imply it is homeomorphic to a closed subspace of
  the standard Cantor set (see \cite[Proposition 5]{Ric63}). We will often
  view $E$ as this subspace (and $F$ as a further closed subspace).

  \begin{definition}

    An ends space $(E,F)$ is called \emph{self-similar} if for any
    decomposition of $E = E_1 \sqcup E_2 \sqcup \cdots \sqcup E_n$ into pairwise
    disjoint clopen sets, then there exists some clopen set $D$ contained in
    some $E_i$ such that $(D,D\cap F)$ is homeomorphic to $(E,F)$. 
  
  \end{definition}

  Following \cite{MR21}, given an ends space $(E,F)$, define a preorder
  $\preceq$ on $E$ where for $x,y \in E$, we say $x \preceq y$ if every
  neighborhood $U$ of $y$ contains some homeomorphic copy of a neighborhood
  $V$ of $x$. \textit{Here and throughout the paper, a neighborhood in an
  end space will always be a clopen neighborhood.} We say $x$ and $y$ are
  equivalent, and write $x \sim y$, if $x \preceq y$ and $y \preceq x$.
  This defines an equivalence relation on $E$. For $x \in E$, denote by
  $E(x)$ the equivalence class of $x$, and denote by $[E]$ the set of
  equivalence classes. From this we get a partial order $\prec$ on $[E]$,
  defined by $E(x) \prec E(y)$ if $x \preceq y$ and $x \nsim y$. Note that
  by definition, if $x \preceq y$, then $y$ is either locally homeomorphic
  to $x$ or an accumulation point of homeomorphic images of $x$ under
  $\Homeo(E,F)$. One easily verifies that, since $F \subset E$ is closed,
  either $E(x) \cap F = \emptyset$ or $E(x) \cap F = E(x)$. Note additionally
  that if there is a homeomorphism $(E, F) \to (E, F)$ which maps $x \mapsto y$,
  then $x \sim y$. Consequently, self-homeomorphims of $(E, F)$ preserve each equivalence class. 
  
  We say a point $x \in E$ is \emph{maximal} if $E(x)$ is maximal with
  respect to $\prec$. Denote by $M(E)$ the set of
  maximal elements in $E$. 

  \begin{proposition}[\cite{MR21}] \label{prop:orderends}

    Let $(E,F)$ be an ends space. The following statements hold.  
    \begin{itemize}
      \item The set $M(E)$ of maximal elements under the partial order
        $\prec$ is non-empty.  
      \item For every $x \in M(E)$, its equivalence class $E(x)$ is either
        finite or homeomorphic to a Cantor set.
      \item If $(E,F)$ is self-similar, then $M(E)$ is a single equivalence
        class $E(x)$, and $E(x)$ is either a singleton or homeomorphic to a
        Cantor set. 
    \end{itemize}

  \end{proposition}
  
  Observe that when $M(E)$ is a single equivalence class $E(x)$ and $F \ne
  \emptyset$, then $E(x) \cap F = E(x)$. 

  \subsection{Classification of infinite-type surfaces}

  By a \emph{surface} we always mean a connected, orientable $2$--manifold.
  A surface has \emph{finite type} if its fundamental group is finitely
  generated; otherwise, it has infinite type. In this paper, we are
  primarily interested in surfaces of infinite type. We refer to
  \cite{Ric63} for details.
  
  The collection of compact sets on a surface $\Sigma$ forms a directed set
  by inclusion. The \emph{space of ends} of $\Sigma$ is \[ E(\Sigma) =
  \lim_{\longleftarrow} \pi_0(\Sigma \setminus K), \] where the inverse
  limit is taken over the collection of compact subsets $K \subset \Sigma$.
  Equip each $\pi_0(\Sigma \setminus K)$ with the discrete topology. Then
  the limit topology on $E(\Sigma)$ is a totally disconnected, compact, and
  metrizable. An element of $E(\Sigma)$ is called an \emph{end} of
  $\Sigma$. 
  
  An end $e \in E(\Sigma)$ is \emph{accumulated by genus} if for all
  subsurface $S \subset \Sigma$ with $e \in E(S)$, then $S$ has infinite
  genus; otherwise, $e$ is called \emph{planar}. Let $E^g(\Sigma)$ be the
  subset of $E(\Sigma)$ consisting of ends accumulated by genus. This is
  always a closed subset of $E(\Sigma)$, with $E(\Sigma)=\emptyset$ if and
  only if $\Sigma$ has finite genus. Hence the pair
  $(E(\Sigma),E^g(\Sigma))$ forms an ends space. Conversely, by
  \cite[Theorem 2]{Ric63}, every ends space $(E,F)$ can be realized as the
  space of ends of some surface $\Sigma$, with
  $(E,F)=(E(\Sigma),E^g(\Sigma))$.

  Infinite type surfaces are completely classified by the following data:
  the genus (possibly infinite), and the homeomorphism type of the ends
  space $(E(\Sigma), E^g(\Sigma))$. More precisely: 
  
  \begin{theorem}[{\cite{Ker23} \cite[Theorem 1]{Ric63}}] \label{thm:classification}
    
    Suppose $\Sigma$ and $\Sigma'$ are boundaryless surfaces. Then, $\Sigma$ and $\Sigma'$ are
    homeomorphic if and only if they have the same (possibly infinite) genus
	and there is a homeomorphism between
    $(E(\Sigma),E^g(\Sigma))$ and $(E(\Sigma'),E^g(\Sigma'))$.

  \end{theorem}

  We remark that although Richards' classification of infinite type
  surfaces is only stated for boundaryless surfaces, it easily extends to
  surfaces with finitely many compact boundary components. That is, two
  surfaces with the same genus, same number of (finitely many) compact
  boundary components, and homeomorphic end space pairs $(E, E^g)$ are in
  fact homeomorphic.

  Fix an orientation on a surface $\Sigma$ and set $(E,F) =
  (E(\Sigma),E^g(\Sigma))$. Let $\Homeo^+(\Sigma)$ be the group of
  orientation preserving homeomorphisms of $\Sigma$. This is a topological
  group equipped with the compact open topology, and moreover it is a
  Polish group. The connected component of the identity is a closed normal
  subgroup $\Homeo_0(\Sigma)$ comprised of homeomorphisms isotopic to the
  identity. The quotient group \[\mcg(\Sigma) = \Homeo^+(S)/\Homeo_0(S)\]
  is called the \emph{mapping class group} of $\Sigma$. When $\Sigma$ has
  finite type, then $\mcg(\Sigma)$ is discrete and finitely presented. When
  $\Sigma$ has infinite type, then $\mcg(\Sigma)$ is a non locally-compact
  Polish group. 
  
  Every homeomorphism of $\Sigma$ induces a homeomorphism of its ends space
  $(E,F)$, and two homotopic homeomorphisms of $\Sigma$ induce the same map
  on $(E,F)$. This gives a continuous homomorphism $\Phi : \Homeo^+(\Sigma)
  \to \Homeo(E,F)$ that factors through $\mcg(\Sigma)$. By \cite{Ric63},
  the map $\Phi$ is also surjective. 
  
  As noted in \cite[Section 4]{MR21}, we also know the preorder $\preceq$
  on $E$ is equivalent to: $x \preceq y$ if and only if for every
  neighborhood $U$ of $y$ there is a neighborhood $V$ of $x$ and $f \in
  \Homeo^+(\Sigma)$ such that $\Phi(f)(V) \subset U$. 

  \begin{definition}
    
    A surface $\Sigma$ is called \emph{self-similar} if its space of ends
    $(E(\Sigma),E^g(\Sigma))$ is self-similar and $\Sigma$ has genus 0 or
    infinite genus. 

  \end{definition}
  
  Note that when $\Sigma$ is self-similar and has infinite genus, then each
  maximal end of $E(\Sigma)$ must be accumulated by genus.

  \begin{remark}
    
    We point out our definition of self-similar surfaces is equivalent to
    another notion. First, following \cite{MR21}, a subset $A$ of a surface
    $\Sigma$ is called \emph{non-displaceable} if $f(A) \cap A \ne
    \emptyset$ for every $f \in \Homeo(S)$. Then, $\Sigma$ is self-similar
    if and only if $\Sigma$ has self-similar ends space and no
    non-displaceable compact subsurfaces. One direction is clear: if
    $\Sigma$ has finite positive genus, then $\Sigma$ has a compact
    non-displaceable subsurface. The other direction is observed by
    \cite[Lemma 5.9 and 5.13]{APV21}.

  \end{remark}

  \subsection{Stable neighborhoods of ends and self-similarity}
  
  \label{sec:stable}

  We now collect some facts about self-similar ends spaces. The key take
  away of this section is that self-similar ends spaces with infinitely
  many maximal ends behave very much like a Cantor set.
  
  \begin{definition} 
    
    Given $x \in E$, a neighborhood $U$ of $x$ is called \emph{stable} if
    any smaller neighborhood $V \subset U$ contains a homeomorphic copy of
    $U$. (Recall that this means that $(V, V \cap F)$ contains a
    homeomorphic copy of $(U, U\cap F)$). 
  
  \end{definition}

  \begin{lemma}[{\cite[Lemma 5.4]{APV21}}] \label{lemma:Eisstablenbhd} 

    If $(E,F)$ is self-similar, then for all maximal element $x \in E$, the
    set $E$ is a stable neighborhood of $x$. 
  
  \end{lemma}

  The following statement is reminiscent of the statement of \cite[Lemma
  4.17]{MR21}, but stronger than what the latter implies, though our proof
  is modeled after theirs.  

  \begin{lemma} \label{lemma:allclopenshomeo} 
    
    Suppose $(E,F)$ is self-similar. Then for all maximal points $x, y \in
    M(E)$ and all clopen neighborhoods $U, V$ resp.\ of $x, y$, there exists
    a homeomorphism $\varphi: (U, U \cap F) \to (V,V \cap F)$ such that
    $\varphi(x) = y$.

  \end{lemma}

  \begin{proof}

    The proof follows a back-and-forth argument. As usual, we will suppress
    $F$, so all maps below are maps of pairs relative to $F$. 
    
    Let $U_0 = U$ and $V_0 = V$. We define the homeomorphism from $U$ to
    $V$ inductively on clopen subsets exhausting $U \setminus \{x \}, V
    \setminus \{y\}$. For convenience, we choose some metric on $E$. We
    choose $U_1 \subseteq U_0$ to be a proper neighborhood of $x$ of
    diameter less than $1$. Since $E$ is a stable neighborhood of $y$ by
    Lemma \ref{lemma:Eisstablenbhd}, and $U_0 \setminus U_1$ is clopen,
    there is a continuous map $$ f_0: U_0\setminus U_1 \to V_0$$ which is a
    homeomorphism onto a clopen image. We can choose $f_0$ such that
    $\image(f_0) \subseteq V_0 \setminus \{y\}$ for the following reasons.
    If $M(E) = \{y\}$, this is automatic. If $M(E)$ is a Cantor set,
    then $V_0 \setminus \{y\}$ contains some $z_0 \in M(E)$, and Lemma
    \ref{lemma:Eisstablenbhd} ensures we can map $U_0 \setminus U_1$
    homeomorphically into a sufficiently small neighborhood of $z_0$ which
    avoids $y$. Since $\image(f_0)$ is clopen, we can choose a proper
    clopen subset $V_1 \subseteq V_0 \setminus \image(f_0)$ of $y$ which
    has diameter less than $1$. By the same token, we can find a map $$
    g_0: V_0 \setminus (V_1 \cup \image(f_0)) \to U_1 \setminus \{x\} $$
    which is a homeomorphism onto a proper clopen image. We similarly
    define a proper clopen neighborhood $U_2 \subseteq U_1 \setminus
    \image(g_0)$ of $x$ which has diameter less than $\frac{1}{2}$.
    
    Inductively, suppose $U_0, \dots, U_{n+1}, V_0, \dots, V_n$ have been
    constructed along with maps that are homeomorphic onto their image
    $$f_i : U_i \setminus (U_{i+1} \cup \image(g_{i-1})) \to V_i \setminus
    V_{i+1}$$ $$g_i : V_i \setminus (V_{i+1} \cup \image(f_i)) \to U_{i+1}
    \setminus U_{i+2}$$ for $0 \leq i \leq n-1$. Using Lemma
    \ref{lemma:Eisstablenbhd} as above, we then define a map which is a
    homeomorphism onto its image $$f_n : U_n \setminus (U_{n+1} \cup
    \image(g_{n-1})) \to V_n \setminus \{y\}$$ and choose a proper clopen
    neighborhood $V_{n+1} \subseteq V_n \setminus \image(f_n)$ of $y$, of
    diameter less than $\frac{1}{n+1}$. Similarly, we define a map which is
    a homeomorphism onto its image $$g_n : V_n \setminus (V_{n+1} \cup
    \image(f_n)) \to U_{n+1} \setminus \{x\}$$ and choose a proper clopen
    neighborhood $U_{n+2} \subseteq U_{n+1} \setminus \image(g_n)$ of $x$,
    of diameter less than $\frac{1}{n+2}$. We thereby inductively construct
    such a sequence of maps $f_0, f_1, \dots$ and $g_0, g_1, \dots$.
    
    Now, restrict target spaces of $f_i, g_i$ to their images. Then, by
    construction, the domains and images of the $f_i$ and $g_i^{-1}$ are
    disjoint and their respective unions are $U \setminus \{x\}$ and $V
    \setminus \{y\}$. Thus, by taking the union of $f_i$ and $g_i^{-1}$, we
    obtain a continuous bijection $\psi: U \setminus \{x\} \to V \setminus
    \{y\}$ since their domains are open subsets. Similarly, we can define
    the continuous inverse of $\psi$ with the $f_i^{-1}$ and $g_i$.
    Moreover, we can extend $\psi$ to a homeomorphism $\varphi: U \to V$ by
    mapping $x$ to $y$. \qedhere 
   \end{proof}
   
\section{Generation of the homeomorphism group }
   
  \label{sec:genunmarked}
   
  Our proof of Theorem \ref{introthm1} in the case of an unmarked surface
  proceeds via the following steps. First, we define the notion of a
  \emph{half-space} of $\Sigma$ and show that the normal closure of a
  single involution contains an $H$-translation for some half-space $H$.
  Formally, if $H$ is a half-space, we say a homeomorphism $\varphi$ is an
  \textit{$H$-translation} if $\{\varphi^n(H)\}_{n \in \Z}$ are all
  pairwise disjoint. We then show that the normal closure of such a
  $\varphi$ contains all homeomorphisms supported on $H$ and that
  half-space supported homeomorphisms generate $\Homeo^+(\Sigma)$.
   
  \begin{definition}
   
    A self-similar ends space $(E,F)$ is \emph{uniformly} self-similar if
    $M(E)$ is one equivalence class homeomorphic to a Cantor set. A surface
    $\Sigma$ is called \textit{uniformly self-similar} if
    $(E(\Sigma),E^g(\Sigma))$ is uniformly self-similar and $\Sigma$ has
    genus 0 or infinity.

  \end{definition}
   
  \begin{definition} \label{def:halfspaces}
    
    For a uniformly self-similar surface $\Sigma$, we will define a
    \textit{half-space} to be a subsurface $H \subset \Sigma$ such that
    \begin{itemize}
      \item[(i)] $H$ is a closed subset of $\Sigma$.
      \item[(ii)] $H$ has a single connected, compact boundary component. 
      \item[(iii)] $E(H)$ and $E(\overline{H^c})$ both contain a maximal
        end of $E(\Sigma)$.
    \end{itemize}   

  \end{definition}
   
  We state the following useful lemma.  
  
  \begin{lemma}[Lemma 2.1 \cite{FGM21}] \label{lemma:clopensfromsurfs}
     
    Let $\Sigma$ be a surface. Every clopen set $U$ of $E(\Sigma)$ is
    induced by a connected subsurface of $\Sigma$ with a single boundary
    circle. Consequently, if $\Sigma$ is uniformly self-similar, and both
    $U, U^c$ contain maximal ends, then this subsurface is a half-space.
    
  \end{lemma}
  	
  The following corollary follows easily from the above Lemma and the fact
  that $E(\Sigma)$ is a subspace of a Cantor set. 

  \begin{corollary} \label{cor:nestedgoodhs}
    
    Let $\Sigma$ be a uniformly self-similar surface, and let $x \in M(E)$.
    There exists a sequence of nested half-spaces $S_1 \supseteq S_2
    \supseteq \dots$ such that $\{x\} = \displaystyle \bigcap_i E(S_i)$ and
    $\partial S_i$ is compact and connected for all $i$.

  \end{corollary}
   
  \begin{lemma} \label{lem:involhavetrans}

    Let $\Sigma$ be a uniformly self-similar surface. Then, there exists a
    half-space $H \subseteq \Sigma$, an involution $\tau$, and $\varphi \in
    \Homeo^+(\Sigma)$ such that $\set{\varphi^n(H)}_{n \in \Z}$ are all
    pairwise disjoint and $\varphi \in \langle \langle \tau \rangle
    \rangle$. Moreover, $\varphi$ is a product of two conjugates of $\tau$,
    and we can choose $\tau$ and $\varphi$ to fix some point in the
    complement of $\displaystyle \bigcup_{n \in \Z} \varphi^n(H)$.

  \end{lemma}

  \begin{proof}
   
    We first construct a somewhat explicit surface which is homeomorphic to
    $\Sigma$. Let $E = E(\Sigma)$ and $M=M(E)$. Let $y, z \in M$ be
    distinct points. Since $E$ is homeomorphic to a subspace of the Cantor
    set, we can find disjoint clopen subsets $\set{U_i \mid i \in \Z}$ such
    that

    \begin{itemize}
      \item $U_i \cap M \neq \emptyset$
      \item $E = \set{y, z} \cup \bigcup_i U_i$
      \item $y$ is an accumulation point of $\{U_i \mid i \leq 0\}$ but not
       $\{U_i \mid i \geq 0\}$, and $z$
      is an accumulation point of $\set{U_i \mid i \geq 0}$ but not $\{U_i \mid i \leq 0\}$.
    \end{itemize}

    By Lemma \ref{lemma:clopensfromsurfs}, there is a half-space $\Sigma_0
    \subseteq \Sigma$ where $E(\Sigma_0) = U_0$. We let $S_i$ be a copy of
    $\Sigma_0$ for each $i \in \Z$, and we let $\underline{S}$ be the
    (oriented) infinite cylinder with countably many disjoint open discs
    removed in a periodic fashion. (We make sure to choose discs with
    disjoint closures.) See Figure \ref{fig:cylinder}. Let $\set{C_i \mid i
    \in \Z}$ be the boundary components of $\underline{S}$. Let $S$ be the
    surface obtained by gluing $C_i$ to $\partial S_i$ via some
    homeomorphism $\psi_i: C_i \to \partial S_i$ which respects orientation
    of the surfaces.
    
    \begin{figure}[htp!] 
      \begin{center}
       \includegraphics{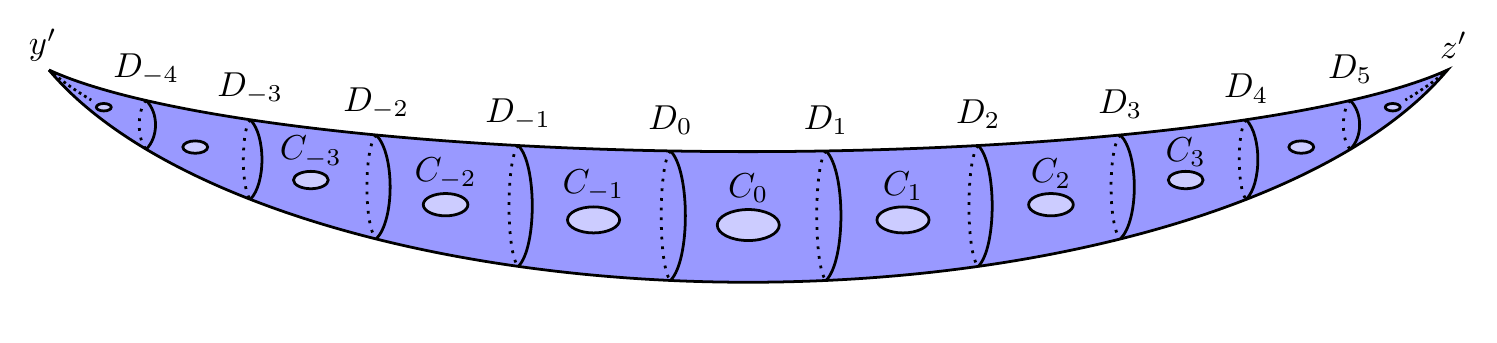} 
       \end{center} 
       \caption{The surface $\underline{S}$.}
      \label{fig:cylinder} 
    \end{figure}

    We first claim that $S \cong \Sigma$. By Theorem
    \ref{thm:classification}, we need only prove that $S, \Sigma$ have the
    same genus and that there is a homeomorphism of end spaces mapping
    $E^g(S)$ to $E^g(\Sigma)$. We will implicitly use some results from
    \cite{Ric63} without referencing them. Recall that $\Sigma$ has genus
    $0$ or $\infty$, and in the latter case, a maximal end must be
    accumulated by genus. Thus $S_i$ and $S$ have infinite genus if and
    only if $\Sigma$ does. By Lemma \ref{lemma:allclopenshomeo}, we have
    that $E(S_i) \cong U_0 \cong U_i$ (respecting genus accumulation). By
    construction of $S$, all of $E(S_i)$ are clopen subsets of $E(S)$.
    Moreover, for $i \in \Z$, let $D_i$ be disjoint curves in the cylinder
    $\underline{S}$ such that they are all translates of each other,
    separate the two ends of the cylinder, and the two sides of $D_i$
    contain $\{C_j \mid j < i\}$ and $\set{C_j \mid j \geq i}$. See Figure \ref{fig:cylinder}.
    Let $P_i^+,
    P_i^-$ be the subsurfaces of $S$ on either side of $D_i$. Then, $P_i^+$
    for $i \geq 0$ (resp. $P_i^-$ for $i \leq 0$) defines an end $z'$
    (resp. $y'$) of $S$. Then, for $n \in \N$, $$E(S) = E(P_{-n}^-) \cup
    E(P_{n}^+) \cup \bigcup_{i=-n}^{n-1} E(S_i),$$ and since only $y', z'$
    are in all $E(P_{-n}^-)$ and $E(P_n^+)$ resp., we have $E(S) = \{y',
    z'\} \cup \bigcup_i E(S_i)$. Moreover, it's clear that $E(S_i)$
    accumulate to $y'$ but not $z'$ as $i \to -\infty$ and $E(S_i)$ accumulates to $z'$
    but not $y'$
    as $i \to \infty$. Therefore, we may define a homeomorphism $E(S) \to
    E(\Sigma)$ mapping $E(S_i) \to U_i$ and $\{y', z'\} \to \{y, z\}$.

    Recall that the homeomorphism $E(S_i) \cong U_i$ maps ends accumulated
    by genus to ends accumulated by genus. If $S$ has infinite genus, then
    every $S_i$ has infinite genus and so $y', z'$ are accumulated by genus (as
    must $y, z$ as they are maximal in $E(\Sigma)$). Consequently, $E(S)
    \cong E(\Sigma)$ maps $E^g(S)$ to $E^g(\Sigma)$ and only $E^g(S)$ to
    $E^g(\Sigma)$. Consequently, $S \cong \Sigma$.

    We now construct an explicit involution $\tau$ of $S$ which normally
    generates the desired $\varphi$. First, we define an involution
    $\underline{\tau}$ on $\underline{S}$. We simply take the ``rotation''
    about an axis piercing $D_0$ and interchanging the ends $y', z'$. This
    induces a homeomorphism between pairs of curves $C_i$ and  $C_j$, where
    $j \ge 0$ and $i=-(j+1)$. For such a pair $i < j$ related by
    $\underline{\tau}(C_i) = C_j$, define a homeomorphism $\tau_{i, j}: S_i
    \to S_j$ such that
    \[\tau_{i, j}|_{\partial S_i} = \psi_j \circ \underline{\tau} \circ
    \psi_i^{-1}.\] For the same pair, define $\tau_{j, i}: S_j \to S_i$ as
    the inverse of $\tau_{i, j}$. Note that
    \[\tau_{j, i}|_{\partial S_j} = \psi_i \circ \underline{\tau}^{-1}
    \circ \psi_j^{-1} = \psi_i \circ \underline{\tau} \circ \psi_j^{-1}\]
    since $\underline{\tau}$ has order $2$. Thus, the $\tau_{i, j}$ agree
    on the overlap with $\underline{\tau}$, and so we extend
    $\underline{\tau}$ to a homeomorphism $\tau$ on all of $S$ via the
    $\tau_{i, j}$. It is clear that $\tau$ has order $2$.
       
    We can similarly define an involution $\sigma$ which is a ``rotation''
    with axis piercing $D_1$. Then, $\sigma(\tau(S_i)) = S_{i+2}$. I.e.
    $\varphi = \sigma \circ \tau$ is the desired $H$-translation where $H =
    S_0$. This establishes the lemma.
       
    To show that we can choose $\tau$ and $\varphi$ to also fix a point
    outside the $S_i$, we do the following. We homotope $D_0$ and $D_1$
    within $\underline{S}$ towards each other until they meet tangentially
    at one point. One can choose an involution $\tau$ that permutes the
    $C_i$ in the same manner as above, maps $D_0$ to itself and fixes the
    common point of $D_0 \cap D_1$. Similarly, $\sigma$ may be chosen to
    map $D_1$ to itself and fix the common point of $D_0 \cap D_1$. Since
    the new $\tau$ and $\sigma$ each permute the $C_i$ in the same manner
    as before, the rest of the argument goes through, and $\varphi$ will
    fix the same point. \qedhere
     
   \end{proof}
   
   \begin{remark} \label{rem:nondisplace}

     Note that in the above proof, the translation $\varphi$ (in the
     version without a fixed point) verifies that a surface $\Sigma$ with
     uniformly self-similar ends space with 0 or infinite genus has no
     non-displaceable surfaces. This also follows from \cite[Lemma 5.9,
     Lemma 5.13]{APV21} which prove it in the case where $E(\Sigma)$ is
     merely self-similar, but our construction gives a different
     perspective.

   \end{remark}
 
   We now show that an $H$-translation normally generates $\Homeo(H,
   \partial H) < \Homeo^+(\Sigma)$ for some half-space $H$. The proof
   technique is sometimes referred to as a ``swindle''.
   
   \begin{lemma} \label{lem:swindle}
  
     Let $\Sigma$ be uniformly self-similar, and let $\varphi$ have the
     properties as described in Lemma \ref{lem:involhavetrans}. Then,
     $\langle \langle \varphi \rangle \rangle$ contains $\Homeo(H,
     \partial H)$. Moreover, every element of $\, \Homeo(H,\partial H)$
     is a product of $\varphi$ and a conjugate of $\varphi^{-1}$.

   \end{lemma}

   \begin{proof}

      Let $f \in \Homeo(H, \partial H) \subseteq \Homeo^+(\Sigma)$. We let
      $\hat{f} = \prod_{i=0}^\infty \varphi^{-i} f \varphi^i$. This is
      well-defined since $\varphi^{-i} f \varphi^i$ is supported on
      $\varphi^{-i}(H)$ and these are pairwise disjoint for all $i \geq 0$
      by assumption. Then, we have the following computation, which, again,
      is valid because of disjoint supports. $$ \left[\hat f,
      \varphi^{-1}\right] = \hat f \varphi^{-1} \hat f^{-1} \varphi =
      \left(\prod_{i=0}^\infty \varphi^{-i} f \varphi^i\right) \varphi^{-1}
      \left(\prod_{i=0}^\infty \varphi^{-i} f^{-1} \varphi^i\right) \varphi
      = \left(\prod_{i=0}^\infty \varphi^{-i} f \varphi^i\right)
      \left(\prod_{i=1}^\infty \varphi^{-i} f^{-1} \varphi^{i}\right) = f.
      \qedhere$$
      
   \end{proof}
   
   \subsection{Half-space homeomorphisms generate}
   
   \label{section:generation} 

   Our proof that homeomorphisms of half-spaces generate $\Homeo^+(\Sigma)$
   relies on a few key facts about half-spaces in uniformly self-similar
   surfaces. We record these as lemmas which we will prove below.
   
   \begin{lemma} \label{lemma:halfspaceintersections}

     Let $H_1, H_2 \subset \Sigma$ be two half-spaces. Then, one of $H_1
     \cap H_2^c, H_1^c \cap H_2^c$ contains a half-space.

   \end{lemma}
   
   \begin{lemma} \label{lemma:homeohalfspaces}
     
     If $H_1, H_2$ are two distinct half-spaces contained in a third
     distinct half-space $H_3$ and both are disjoint from a fourth
     half-space $H_4 \subset H_3$, then there exists $\varphi \in
     \Homeo^+(\Sigma)$ supported on $H_3$ such that $\varphi(H_1) = H_2$.   		

   \end{lemma}
   
   \begin{lemma} \label{lemma:selfsimhalfspaces}
    
     If $H \subseteq \Sigma$ is a half-space, then so is $\overline{H^c}$.
     All half-spaces are homeomorphic via an ambient homeomorphism of
     $\Sigma$. Every half-space contains two disjoint half-spaces.

   \end{lemma}
   
   Using these three lemmas we can prove one of our main results, namely,
   that half-space homeomorphisms generate $\Homeo^+(\Sigma)$.
   
   \begin{theorem} \label{thm:halfspacehomeogen}
     
     Let $\Sigma$ be a uniformly self-similar surface, and let $H \subseteq
     \Sigma$ be a half-space. Then, $\Homeo^+(\Sigma)$ is the normal
     closure of the subgroup $\Homeo(H, \partial H)$. Furthermore, every
     element of $\Homeo^+(\Sigma)$ is a product of at most $3$
     homeomorphisms, each of which is conjugate to an element of $\,
     \Homeo(H, \partial H)$.

   \end{theorem}
   
   \begin{proof}

     Let $f \in \Homeo^+(\Sigma)$. First, note that by Lemma
     \ref{lemma:selfsimhalfspaces}, any half-space supported homeomorphism
     is conjugate into $\Homeo(H, \partial H)$. Thus, it suffices to show
     $f$ is a product of at most $3$ half-space supported homeomorphisms.

     Let $H_1 = H$ and $H_2 = f(H)$. We now apply Lemma
     \ref{lemma:halfspaceintersections}, and first consider the case where
     $H_1^c \cap H_2^c$ contains a half-space. By Lemma
     \ref{lemma:selfsimhalfspaces}, $H_1^c \cap H_2^c$ contains two
     disjoint half-spaces $H_3, H_4$. Applying Lemma
     \ref{lemma:homeohalfspaces} to $H_1, H_2, \overline{H_3^c},$ and
     $H_4$, we see that there is a homeomorphism $\varphi_1$, supported on
     $\overline{H_3^c}$, such that $\varphi_1(H_2) = \varphi_1(f(H_1)) =
     H_1$. By further composing by some $\varphi_2$ supported on $H_1$, we
     can ensure $\varphi_2 \circ \varphi_1 \circ f$ restricts to the
     identity on $H_1$. Finally, composing by an appropriate third
     homeomorphism $\varphi_3$ supported on $\overline{H_1^c}$, we obtain
     $\varphi_3 \circ \varphi_2 \circ \varphi_1 \circ f = \text{Id}$. Note
     that $\varphi_2 \circ \varphi_1$ is supported on $\overline{H_3^c}$,
     so in this case we only require two half-space supported
     homeomorphisms.
    
     Now, suppose we are in the case where $H_1 \cap H_2^c$ contains  a
     half-space. By Lemma \ref{lemma:selfsimhalfspaces}, we can assume $H_1
     \cap H_2^c$ contains three disjoint half-spaces $H_3, H_4, H_5$. By
     Lemma \ref{lemma:homeohalfspaces} applied to $H_2, H_4,
     \overline{H_3^c},$ and $H_5$, there exists a homeomorphism $\psi$
     supported on $\overline{H_3^c}$ such that $\psi(H_2) = H_5$. Since
     $H_5 \subset H_1$, the subsurface $H_1^c \cap \psi(f(H_1))^c = H_1^c$
     contains a half-space, and we are reduced to the previous case. In
     this case, we see that $f$ is a product of three half-space supported
     homeomorphisms. \qedhere

   \end{proof}
   
   We now prove the required lemmas.
   
   \begin{proof}[Proof of Lemma \ref{lemma:halfspaceintersections}]

     By definition, $E(\overline{H_2^c})$ contains some maximal end $x$. By
     Corollary \ref{cor:nestedgoodhs}, there exist nested half-spaces $S_1
     \supset S_2 \supset \dots$ such that $\bigcap_i E(S_i) = \{x\}$. Since
     these half-spaces leave every compact set, eventually some $S_i$ does
     not intersect $\partial H_1 \cup \partial H_2$. Since $x \in
     E(\overline{H_2^c})$ and the $S_i$ are connected, either $S_i
     \subseteq H_1 \cap H_2^c$ or $S_i \subseteq H_1^c \cap H_2^c$.
     \qedhere

   \end{proof}
   
   \begin{proof}[Proof of Lemma \ref{lemma:selfsimhalfspaces}]

     The first statement follows immediately from the definition of
     half-space. Since a half-space has a maximal end of $\Sigma$, it has
     the same genus as $\Sigma$. Thus, by the classification of surfaces
     and Lemma \ref{lemma:allclopenshomeo}, any two half-spaces are
     homeomorphic. Since the (closures of) the complements are half-spaces
     too, we can map the complement to the complement and extend the
     homeomorphism to all of $\Sigma$.
    
     By assumption, $E(H)$ has some maximal end $x$ of $E(\Sigma)$. Since
     $E(H)$ is a clopen in $E(\Sigma)$ and the set of maximal ends is a
     Cantor set, $E(H)$ contains another distinct maximal end $y$. Using
     Corollary \ref{cor:nestedgoodhs} for $x$ and $y$ each and compactness
     of boundaries of half spaces, we can easily deduce the existence of
     the required half-spaces. \qedhere

   \end{proof}
   
   \begin{proof}[Proof of Lemma \ref{lemma:homeohalfspaces}]

     The presence of the half-space $H_4$ guarantees that $E(\overline{H_3
     \setminus H_1})$ and $E(\overline{H_3 \setminus H_2})$ both contain a
     maximal end. Thus, by Lemma \ref{lemma:allclopenshomeo},
     $$(E(\overline{H_3 \setminus H_1}), E^g(\overline{H_3 \setminus H_1}))
     \cong (E(\overline{H_3 \setminus H_2}), E^g(\overline{H_3 \setminus
     H_2})).$$ Clearly, the two subsurfaces have the same genus and finite
     number of boundary components, and so $\overline{H_3 \setminus H_1}
     \cong \overline{H_3 \setminus H_2}$. Similarly, by Lemma
     \ref{lemma:selfsimhalfspaces}, $H_1 \cong H_2$. By arranging the
     homeomorphisms to be identical on the overlapping boundary component,
     we produce a homeomorphism $H_3 \to H_3$ mapping $H_1 \to H_2$ and
     $\overline{H_3 \setminus H_1} \to \overline{H_3 \setminus H_2}$.
     \qedhere
     
   \end{proof}
	
   \begin{theorem} \label{thm:unmarkedhomeo}

     If $\Sigma$ is uniformly self-similar, then $\Homeo^+(\Sigma)$

     \begin{itemize}
       \item is normally generated by a single involution,
       \item is normally generated by an $H$-translation,
       \item is uniformly perfect.
      \end{itemize}
      
      Moreover, each element of $\, \Homeo^+(\Sigma)$ is a product of at
      most $3$ commutators, $6$ $H$-translations, and $12$ involutions.

   \end{theorem}

   \begin{proof}

      Combine Lemma \ref{lem:involhavetrans}, Lemma \ref{lem:swindle}, and
      Theorem \ref{thm:halfspacehomeogen}. \qedhere

   \end{proof}
	
   By considering quotients of $\Homeo^+(\Sigma)$ onto the mapping class
   group $\mcg(\Sigma)$ and the homeomorphism group of its ends space
   $(E(\Sigma),E^g(\Sigma))$, we also derive the following corollaries.
   Note that for a half-space $H \subset \Sigma$, $E(H)$ is a clopen set
   containing a non-empty proper subset of $M(E(\Sigma))$. 

   \begin{corollary}

     If $\Sigma$ is uniformly self-similar, then the statements of
     Theorem \ref{thm:unmarkedhomeo} also hold for $\mcg(\Sigma)$.

   \end{corollary}

   \begin{corollary}

     If $(E,F)$ is uniformly self-similar, then the statements of Theorem
     \ref{thm:unmarkedhomeo} also hold for $\Homeo(E,F)$, where a
     half-space $H \subset E$ is a clopen set containing a non-empty
     proper subset of $M(E)$.

   \end{corollary}
   
   \begin{proof}

      By \cite[Theorem 2]{Ric63}, there is a surface $\Sigma$ such that
      $(E(\Sigma), E^g(\Sigma)) \cong (E, F)$ and the genus is $0$ if $F =
      \emptyset$ and infinite if $F \neq \emptyset$. Thus $\Sigma$ is
      uniformly self-similar when $(E,F)$ is. The corollary then follows
      from Theorem \ref{thm:unmarkedhomeo} and surjectivity of
      $\Homeo^+(\Sigma) \to \Homeo(E(\Sigma), E^g(\Sigma))$. \qedhere

   \end{proof}

\section{Surfaces with a marked point}

  \label{sec:genmarked}

  The proof in the case of a marked surface is very similar to the unmarked
  case, and we will use some of the same lemmas. Let $\Sigma$ be a
  uniformly self-similar surface with a fixed basepoint $\ast \in \Sigma$.
  We define half-space exactly as before, but distinguish between
  \textit{marked half-spaces} which contain $\ast$ and \textit{unmarked
  half-spaces} which don't. The main new lemma we require is the following.

  \begin{lemma} \label{lemma:haveDTs}

    Let $\Sigma$ be a uniformly self-similar surface with a marked point
    $\ast \in \Sigma$. Let $H \subseteq \Sigma$ be an unmarked half-space.
    Then, every Dehn twist in $\mcg(\Sigma, \ast)$ is contained in the
    normal closure of $\mcg(H, \partial H)$.

  \end{lemma}
  
  \begin{remark}

    For convenience and simplicity, we will conflate half-spaces and simple
    closed curves with their ambient isotopy classes rel $\ast$ throughout
    this section.

  \end{remark}
  
  \begin{proof}

    Let $T_\gamma \in \mcg(\Sigma, \ast)$ be the Dehn twist about a simple
    closed curve $\gamma$ (which avoids $\ast$). First, we consider the
    case where $\gamma$ is nonseparating. Then, $\Sigma$ has infinite
    genus. Since $\gamma$ is compact, Corollary \ref{cor:nestedgoodhs}
    implies that $\gamma$ is contained in some half-space (or the closure
    of its complement which is also a half-space) which we denote by $H$.
    This case is concluded if $H$ is unmarked. Suppose instead $H$ is
    marked. Then, since $\gamma$ is nonseparating, we can find a path from
    $\partial H$ to $\ast$ which avoids $\gamma$. Deleting some small
    regular neighborhood of this path from $H$, we obtain an unmarked
    half-space containing $\gamma$.
    
    Now, suppose $\gamma$ is a separating curve, and let $S_1, S_2
    \subseteq \Sigma$ be the two surfaces on either side of $\gamma$. If
    both $E(S_1), E(S_2)$ contain a maximal end of $\Sigma$, then they are
    both half-spaces whose mapping class groups contain $T_\gamma$, and one
    must be unmarked. Suppose, w.l.o.g., then that $E(S_1)$ contains no
    maximal ends. If $S_1$ is also unmarked, then we can connect it by some
    strip (avoiding $\ast)$ to an unmarked half-space in $S_2$ to create a
    new unmarked half-space $H'$ which contains $S_1$. Then $\Homeo(H',
    \partial H')$ contains $T_\gamma$.
    
    The difficult case is when $S_1$ contains no maximal ends but is
    marked. Using Corollary \ref{cor:nestedgoodhs} repeatedly, we can find
    three disjoint half-spaces $H_1, H_2, H_3$ contained in $S_2$. Since
    half-spaces have connected boundary, the complement of $H_1 \cup H_2
    \cup H_3 \cup S_1$ is connected, and we may choose disjoint paths
    $\alpha_1, \alpha_2$ in this complement connecting $\gamma = \partial
    S_1$ to $\partial H_1, \partial H_2$ respectively. Let $L$ be a regular
    neighborhood of $\gamma \cup \partial H_1 \cup \partial H_2 \cup
    \alpha_1 \cup \alpha_2$ in this complement. Then, $L$ is a sphere with
    $4$ boundary components, i.e.\ a lantern, where three boundary curves
    are $\gamma, \partial H_1,$ and $\partial H_2$ and the fourth is some
    simple closed curve $\beta$ bounding a half-space $H_4$ containing
    $H_3$.
    
    We seek to use the lantern relation to show that $f$ is a product of
    homeomorphisms supported on an unmarked half-space. The lantern
    relation implies that $T_\gamma$ is equal to a word in the Dehn twists
    about $\partial H_1, \partial H_2, \beta$ and three other simple closed
    curves $\delta_1, \delta_2, \delta_3$ each of which separates $L$ into
    two three-holed spheres. Thus for all $i = 1, 2, 3$, each side of
    $\delta_i$ must contain at least one of $H_1, H_2, H_3$, i.e.\ each
    $\delta_i$ separates $\Sigma$ into one marked and one unmarked
    half-space, and thus $\delta_i$ lies in an unmarked half-space.
    Consequently, the twists about $\partial H_1, \partial H_2, \beta,$ and
    the $\delta_i$ are all supported on an unmarked half-space. The lemma
    follows. \qedhere

  \end{proof}
  
  In the unmarked case, we replace Lemma \ref{lemma:homeohalfspaces} with
  the following.

  \begin{lemma} \label{lemma:homeohalfspacesv2}

    If $H_1, H_2,$ and $H_3$ are disjoint unmarked half-spaces, then there
    is a homeomorphism $\varphi \in \Homeo(\Sigma)$ supported on some
    unmarked half-space $H_4$ such that $\varphi(H_1) = H_2$.   		

  \end{lemma}

  \begin{proof}

    By Lemma \ref{lemma:selfsimhalfspaces}, $H_3$ contains two unmarked
    disjoint half-spaces $H_3', H_3''$. Since half-spaces have single
    boundary components, the complement of $H_1 \cup H_2 \cup H_3' \cup
    H_3''$ is connected, and we can attach $H_1$ to $H_2$ and $H_3'$ by two
    strips disjoint from $H_3''$ and the marked point to create a
    subsurface $H_4$ with a single boundary circle that contains $H_1,
    H_2,$ and $H_3'$. Since both $E(H_4) \supset E(H_1)$ and $E(H_4^c)
    \supset E(H_3'')$ contain a maximal end, $H_4$ is a half-space. We can
    now apply Lemma \ref{lemma:homeohalfspaces} to $H_1, H_2, H_4,$ and
    $H_3'$. \qedhere

  \end{proof}
  
  We can now prove the analogous theorem that half-space supported
  homeomorphisms generate $\mcg(\Sigma, \ast)$.
  
  \begin{theorem} \label{thm:unmarkedhalfspacegen}

    Let $\Sigma$ be a uniformly self-similar surface with a marked point
    $\ast \in \Sigma$, and let $H \subseteq \Sigma$ be an unmarked
    half-space. Then, $\mcg(\Sigma, \ast)$ is generated by the normal
    closure of $\mcg(H, \partial H)$.

  \end{theorem}

  \begin{proof}
    
    Let $f \in \mcg(\Sigma, \ast)$. All unmarked half-spaces are the same
    up to $\Homeo(\Sigma, \ast)$ by an argument nearly identical to that in
    the proof of Lemma \ref{lemma:selfsimhalfspaces}. Therefore, it
    suffices to show $f$ is a product of mapping classes supported on
    unmarked half-spaces.
    
    Let $H_1$ be an unmarked half-space, and let $C$ be a simple closed
    curve such that $C$ and $\partial H_1$ bound an annulus containing
    $\ast$ (in the interior). Let $H_2 = f(H_1)$. Then, by Lemma
    \ref{lemma:halfspaceintersections}, either $H_1 \cap H_2^c$ or $H_1^c
    \cap H_2^c$ contains a half-space, which we can choose to be unmarked,
    by passing to a deeper half-space if necessary. 
    
    We first show that there is some mapping class $g$ in the normal
    closure of $\mcg(H, \partial H)$ such that $g_1(f(H_1)) = H_1$ and $g_1
    \circ f|_{H_1} = \id|_{H_1}$. Let's first consider the case where
    $H_1^c \cap H_2^c$ contains an unmarked half-space. Then, by Lemma
    \ref{lemma:selfsimhalfspaces}, $H_1^c \cap H_2^c$ contains two disjoint
    unmarked half-spaces $H_3, H_4$. By Lemma
    \ref{lemma:homeohalfspacesv2}, there are two mapping classes supported
    on some unmarked half-spaces, one which maps $H_2$ to $H_3$ and another
    which maps $H_3$ to $H_1$. By composing these maps with some
    appropriate third mapping class supported on $H_1$, we obtain the
    desired $g_1$. If instead $H_1 \cap H_2^c$ contains an unmarked
    half-space, then $H_1 \cap H_2^c$ contains two disjoint unmarked
    half-spaces $H_3, H_4$. By Lemma \ref{lemma:homeohalfspacesv2}, there
    is some mapping class $h$ supported on an unmarked half-space such that
    $h(f(H_1)) = h(H_2) = H_3 \subseteq H_1$. Thus $H_1^c \cap H_3^c =
    H_1^c$ contains an unmarked half-space, the one bounded by $C$, and we
    are reduced to the first case.
    
    Let $C' = g_1(f(C))$. Then $C'$ and $\partial H_1$ bound an annulus
    containing $\ast$. (Note that $C'$ need not be $C$ up to ambient
    isotopy fixing $\ast$.) Let $S \subseteq \Sigma$ be a compact
    subsurface with the following properties.
    \begin{itemize}
    \item $S$ contains the annulus bounded by $\partial H_1$ and $C$
    and the annulus bounded by $\partial H_1$ and $C'$.
    \item $S$ does not intersect the interior of $H_1$
    \item no boundary component bounds a disc in $\Sigma$. 	
    \end{itemize}
    Within $S$, each of $C, C'$ is a separating curve which bounds an
    annulus with $\partial H_1 \subset \partial S$ containing a marked
    point. Consequently, the genus of the separating curves $C,
    C'$ must be identical and they partition the boundary of $S$
    identically. Thus, there is some mapping class in $\mcg(S, \partial S
    \cup {\ast})$ mapping $C'$ to $C$. Since $\mcg(S, \partial S \cup
    {\ast})$ is generated by Dehn twists, by Lemma \ref{lemma:haveDTs},
    there is some $g_2$ in the normal closure of $\mcg(H, \partial H)$ such
    that $g_2(g_1(f(C)) = C$ and $g_2 \circ g_1 \circ f|_{H_1} =
    \id|_{H_1}$.
    
    Let $H_0$ be the unmarked half-space bounded by $C$. Clearly,
    $g_2(g_1(f(H_0)) = H_0$. Since the mapping class group of the annulus
    with a marked point between $H_0$ and $H_1$ is generated by Dehn
    twists, by Lemma \ref{lemma:haveDTs}, we can compose by some third
    element $g_3$ in the normal closure of $\mcg(H, \partial H)$ such that
    $g_3 \circ g_2 \circ g_1 \circ f = \id$. \qedhere

  \end{proof}
  
  We can now easily prove the analogous theorem for the mapping class group
  of a marked uniformly self-similar surface. Note that we have no
  statements about uniform perfection, or about a bound on the word length
  of an element as a product of involutions, or about the homeomorphism
  group. The first two are impossible by a result of J. Bavard \cite{Bav16}
  in the case of $S^2$ minus a Cantor set. The proof fails to show every
  element is a word of uniformly bounded length in involutions and
  half-space supported homeomorphisms only because of the step where we map
  $C'$ to $C$. The theorem is only proven for the mapping class group and
  not the homeomorphism group because we use the lantern relation in the
  proof of Lemma \ref{lemma:haveDTs}.
  
  \begin{theorem}

    Let $\Sigma$ be a uniformly self-similar surface with a marked point
    $\ast \in \Sigma$. Then, $\mcg(\Sigma, \ast)$ is generated by
    involutions. Moreover, $\mcg(\Sigma, \ast)$ is normally generated by a
    single involution and is a perfect group.

  \end{theorem}
  
  \begin{proof}

    Let $\tau$ and  $\varphi$ be as in Lemma \ref{lem:involhavetrans}.
    Lemma \ref{lem:swindle} applies equally to $\mcg(\Sigma, \ast)$ (with
    an identical proof), and so $\langle \langle \tau \rangle \rangle$
    contains $\mcg(H, \partial H)$ for some unmarked half-space and all
    elements of $\mcg(H, \partial H)$ are a single commutator in
    $\mcg(\Sigma)$. Theorem \ref{thm:unmarkedhalfspacegen} finishes the
    proof. \qedhere

  \end{proof}

\section{Self-similar but not uniformly}

  \label{sec:bigabel}

  It is natural to wonder whether surfaces with a self-similar ends space
  and with genus $0$ or $\infty$ are generated by involutions, are perfect,
  etc. It is already known that the mapping class group of the one-ended,
  infinite genus surface has abelianization containing an uncountable
  direct sum of $\Q$'s \cite{DD21}. This surface fits into this category,
  but perhaps is not a particularly compelling example since the results of
  \cite{DD21} are for pure mapping class groups of infinite-type surfaces,
  and for the one-ended infinite genus surface, the mapping class group
  happens to coincide with the pure mapping class group. However, using a
  covering trick and some of the results of \cite{DD21}, we can prove that
  the abelianization of $\mcg(\R^2 \setminus \N)$ is similarly large.

  \begin{proposition} \label{prop:flute}
    
    $\mcg(\R^2 \setminus \N)$ surjects onto $\bigoplus_{2^{\aleph_0}} \Q$.

  \end{proposition}

  For the proof of the proposition, we need the following fact about
  abelian groups, which follows from \cite[Theorem 21.3 and 23.1]{Fuc70}.

  \begin{lemma} \label{lem:divisible}

    Let $A$ be an abelian group. Suppose $A$ contains $\bigoplus_I \Q$ for
    some non-empty set $I$. Then $A$ surjects onto $\bigoplus_I \Q$.

  \end{lemma}

  \begin{proof}[Proof of Proposition \ref{prop:flute}]

    Let $\Sigma_L$ be the infinite genus surface with one end. This admits
    a $2$-fold branched cover of $\R^2$ where $\R^2 = \Sigma_L/D$ and $D=
    \Z/2\Z$ acts by an involution. See Figure \ref{fig:branchcover}. More
    formally, one can construct this from gluing infinitely many copies of
    a $2$-fold branch cover of an annulus by a $2$-holed torus and one copy
    of a $2$-fold branched cover of a disc by a $1$-holed torus. Let
    $\Sigma_F$ be $\R^2$ with the branch points removed, i.e. $\Sigma_F
    \cong \R^2 \setminus \N$, and let $\Sigma_{PL}$ be $\Sigma_L$ with the
    branch points removed. Then, $\Sigma_{PL} \to \Sigma_F$ is a regular
    degree $2$ cover with deck group $D$.
    
    \begin{figure} [htp!] 
      \begin{center}
       \includegraphics{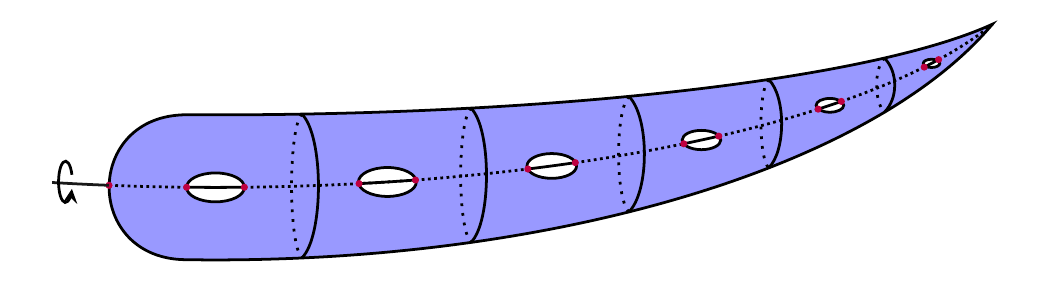} 
       \end{center} 
       \caption{The surface $\Sigma_L$ admits an involution which gives
       a degree 2 branched cover of $\R^2$ branched at the red points.}
      \label{fig:branchcover} 
    \end{figure}

    Choose marked points $\tilde{\ast} \in \Sigma_{PL}$ and $\ast \in
    \Sigma_F$. We first show that there is a lifting homomorphism
    $\mcg(\Sigma_F, \ast) \to \mcg(\Sigma_{PL}, \tilde{\ast})$, defined by
    lifting representative homeomorphisms. Since these are mapping class
    groups fixing marked points, it is a straightforward consequence of
    covering space theory that such a homomorphism exists and is unique
    provided the action of $\mcg(\Sigma_F, \ast)$ preserves the subgroup $K
    = \ker(\pi_1(\Sigma_F, \ast) \to D)$. Since the punctures of $\Sigma_F$
    came from branched points of degree $2$, any simple loop in
    $\pi_1(\Sigma_F, \ast)$ which encloses a single puncture does not lift
    to a closed curve in $\Sigma_{PL}$ and so must map to the nontrivial
    element of $D$. We can choose a basis $\{\beta_n\}_{n \in \N}$ of the
    free group $\pi_1(\Sigma_F, \ast)$ consisting entirely of simple loops
    each enclosing a single puncture. Consequently, $K$ consists precisely
    of those even length words in this generating set. For any mapping
    class $f \in \mcg(\Sigma_F, \ast)$, the set $\{f(\beta_n)\}_{n \in \N}$
    is also another generating set of simple loops enclosing single
    punctures, and for the same reasons $K$ consists of words of even
    length in these generators. It is clear then that $f(K) = K$.
    
    We now have a lifting homomorphism $\mcg(\Sigma_F, \ast) \to
    \mcg(\Sigma_{PL}, \tilde{\ast})$. Since the points deleted from
    $\Sigma_L$ are isolated, $\mcg(\Sigma_{PL}, \tilde{\ast})$ preserves
    that set of ends, and we have a well-defined forgetful map
    $\mcg(\Sigma_{PL}, \tilde{\ast}) \to \mcg(\Sigma_L, \tilde{\ast})$. In
    \cite{DD21}, explicit mapping classes are constructed which project to
    nontrivial elements in the abelianization of $\mcg(\Sigma_L,
    \tilde{\ast})$. (See \cite[Theorem 6.1]{DD21}.) Specifically, if
    $\{\gamma_n\}_{n \in \N}$ is a sequence of distinct, pairwise disjoint,
    separating, simple closed curves where each $\gamma_n$ separates the
    marked point from the single end of $\Sigma_L$, then the subgroup
    topologically generated by the twists $\{T_{\gamma_n}\}_{n \in \N}$
    projects to a group containing a $\bigoplus_{2^{\aleph_0}} \Q$. One can
    easily find such $\gamma_n$ which double cover simple closed curves
    $\alpha_n$ in $\Sigma_F$, and so $T_{\gamma_n}$ is the lift of
    $T_{\alpha_n}^2$. (E.g. one can choose $\alpha_1$ to be a simple closed
    curve bounding a disc with the marked point and three punctures
    and then choose $\alpha_i, \alpha_{i+1}$ to always bound an annulus
    with two punctures. Then $\gamma_i$ are the preimages of the $\alpha_i$
    under the covering map.) Thus, $\mcg(\Sigma_F, \tilde{\ast})$ maps onto the
    same abelian group (generated by the $T_{\gamma_n}$). I.e.
    $\HH_1(\mcg(\Sigma_F, \tilde{\ast}); \Z)$ has a quotient $A$ containing
    $\bigoplus_{2^{\aleph_0}} \Q$. By Lemma \ref{lem:divisible}, $A$ maps
    onto $\bigoplus_{2^{\aleph_0}} \Q$, so we also get a surjection
    $\varphi: \HH_1(\mcg(\Sigma_F, \tilde{\ast}); \Z) \to
    \bigoplus_{2^{\aleph_0}} \Q$.

    To pass to $\mcg(\Sigma_F)$, we borrow a technique from \cite{DD21}.
    Consider the Birman short exact sequence (see \cite{DD21}) 
    $$ 1 \longrightarrow
    \pi_1(\Sigma_F, \ast) \longrightarrow \mcg(\Sigma_F, \tilde{\ast})
    \longrightarrow \mcg(\Sigma_F) \longrightarrow 1.$$
    Abelianization is right exact, so we get the commutative diagram 
    $$ 
    \begin{tikzcd}
      \bigoplus_{_{\aleph_0}} \Z \arrow[r] \arrow[d,"\id"] 
      & 
      \HH_1(\mcg(\Sigma_F, \tilde{\ast}); \Z) \arrow[r] \arrow[d, "\varphi"] 
      &
      \HH_1( \mcg(\Sigma_F); \Z) \arrow[r] \arrow[d,"\bar\varphi"] & 0\\
      \bigoplus_{_{\aleph_0}} \Z \arrow[r] 
      & 
      \bigoplus_{2^{\aleph_0}} \Q \arrow[r]
      &
      P \arrow[r]
      & 0
    \end{tikzcd}
    $$ 
    
    The image of $\bigoplus_{_{\aleph_0}} \Z$ in $\bigoplus_{2^{\aleph_0}}
    \Q$ still misses a copy of $\bigoplus_{2^{\aleph_0}} \Q$, so the
    quotient $P$ still contains a copy of $\bigoplus_{2^{\aleph_0}} \Q$.
    The map $\bar\varphi$ is surjective, so we can conclude $\HH_1(
    \mcg(\Sigma_F); \Z)$ surjects onto $\bigoplus_{2^{\aleph_0}} \Q$, again
    by Lemma \ref{lem:divisible}. \qedhere

  \end{proof}

  We now produce many more classes of examples by building surfaces that
  naturally map onto the one-ended infinite genus surface $\Sigma_L$ or
  $\Sigma_F = \R^2 \setminus \N$. 

  \begin{theorem} \label{thm:bigabel}

    Suppose $\Sigma$ is a surface of one of the following two types. 
    \begin{itemize}
      \item[(1)] $E(\Sigma)$ has exactly one end accumulated by genus.
      \item[(2)] $\Sigma$ has genus $0$ and $E(\Sigma)$ has
        one maximal end $y$, such that in the partial order on $[E]$, the
        class of $y$ has an immediate predecessor $E(x)$ with countably
        infinite cardinality. 
    \end{itemize}

    Then $\mcg(\Sigma)$ maps onto $\bigoplus_{2^{\aleph_0}} \Q$. 

  \end{theorem}

  \begin{proof}

    Choose a marked point $\ast$ on $\Sigma$. Note that by the same trick
    of using the Birman short exact sequence $$ 1 \longrightarrow
    \pi_1(\Sigma, \ast) \longrightarrow \mcg(\Sigma, \ast) \longrightarrow
    \mcg(\Sigma) \longrightarrow 1,$$ it is enough to show the
    abelianization of $\mcg(\Sigma,\ast)$ maps onto
    $\bigoplus_{2^{\aleph_0}} \Q$.

    The statement for the one-ended infinite genus surface $\Sigma_L$ is by
    \cite{DD21}. The statement for $\Sigma_F = \R^2 \setminus \N$ is
    Proposition \ref{prop:flute}. For all other case, we will consider an
    appropriate map to one of these two surfaces.

    On $\Sigma_L$, we will say a sequence of simple closed curves
    $\set{\gamma_n}_{n \in \N}$ is \emph{good} if the curves are distinct,
    pairwise disjoint, separating, and each curve separates the maximal end
    of $\Sigma_L$ from the marked point. On $\Sigma_F$, a sequence of
    curves $\set{\alpha_n}_{n \in \N}$ is \emph{good} if under the covering
    map $(\Sigma_L,\ast) \to (\Sigma_F,\ast)$, each $\alpha_n$ is double
    covered by a curve $\gamma_n$ and the sequence $\{\gamma_n\}$ is good.
    By \cite{DD21} and the proof of Proposition \ref{prop:flute}, the
    subgroup topologically generated by Dehn twists about a good sequence
    of curves maps onto $\bigoplus_{2^{\aleph_0}} \Q$ under the map to the
    abelianization of the mapping class group.

    First, assume $\Sigma$ is of the first type. The proof in the other
    case will be similar. The assumption on $\Sigma$ means we have a map
    $(\Sigma,\ast) \to (\Sigma_L,\ast)$ by forgetting all but the only end
    accumulated by genus. This induces a well-defined map
    $\mcg(\Sigma,\ast) \to \mcg(\Sigma_L,\ast)$, since this end is
    invariant under $\mcg(\Sigma, \ast)$. By the previous paragraph, it is
    enough to exhibit a sequence of pairwise disjoint curves
    $\set{\alpha_n}_{n \in \N}$ on $\Sigma$ whose image under the forgetful
    map forms a good sequence on $\Sigma_L$. To do this we will represent
    $\Sigma$ in an explicit way as described below. 
    
    \begin{figure} [htp!] 
      \begin{center}
       \includegraphics{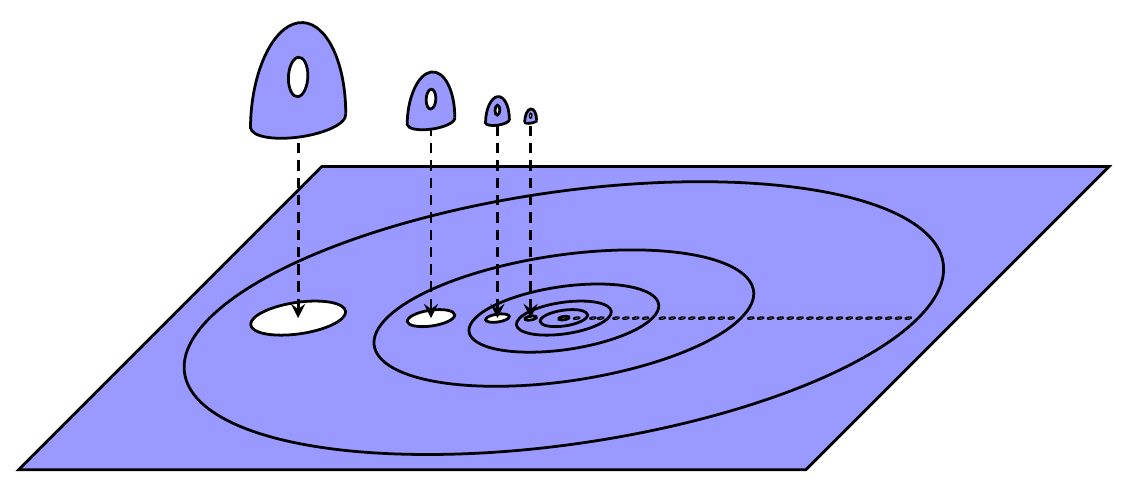} 
       \end{center} 
       \caption{Building $\Sigma$ of the first type.}
       \label{fig:lochends} 
    \end{figure}

    Identify $S^2 = \R^2\cup \set{\infty}$ with base point $\infty$. We
    will construct $\Sigma$ from $S^2$ by removing points from $\R^2
    \subset S^2$ and attaching handles appropriately. Let $K \subset [0,1]$
    be the standard Cantor set. Recall $E(\Sigma)$ is homeomorphic to a
    closed subset of $K$. Since the homeomorphism group of $K$ acts
    transitively, we can realize $E(\Sigma)$ as a closed subset $E \subset
    K$ with the only end accumulated by genus at $0$. By \cite{Ric63}, the
    ends space of $S^2-E$ is homeomorphic to $E$. It remains to attach
    handles to $\R^2 - E$ so that the handles will only accumulate onto the
    origin. To this end, choose a sequence $\set{y_n}_{n \in \N} \subset
    [0,1] - K$ such that $y_n \to 0$ monotonically. In particular,
    $\set{y_n} \cap E = \emptyset$. Let $d_n = y_n-y_{n+1}$. For each $n$,
    let $T_n$ be a torus with one boundary component. Let $z_n$ be the
    midpoint of $[y_{n+1},y_n]$. In $\R^2$, let $p_n = (-z_n, 0)$, and
    $B_n$ be the open ball of diameter $d_n/2$ centered at $p_n$. Now
    remove each $B_n$ from $\R^2$ and attach $T_n$ by gluing $\partial T_n$
    to $\partial B_n$. Let $\Sigma'$ be the resulting surface with marked
    point $\infty$. By construction, the tori accumulate only onto the
    origin. It follows then by the classification of surfaces, $\Sigma'$ is
    homeomorphic to $\Sigma$, and we can make this homeomorphism take
    $\infty$ to $\ast$. By filling in all of $E$ except the origin, we get
    a marked surface $(\Sigma_L',\infty)$ homeomorphic to
    $(\Sigma_L,\ast)$, and a representation of the forgetful map
    $(\Sigma,\ast) \to (\Sigma_L,\ast)$. Using this picture, it is now easy
    to find the curves $\alpha_n$ which we take to be the circle of radius
    $y_n$ centered at the origin. By construction, the circles
    $\{\alpha_n\}$ avoid $E$, are pairwise disjoint, and each separates the
    origin from $\infty$. Furthermore, since there is a handle between two
    consecutive circles $\alpha_n$ and $\alpha_{n+1}$, namely $T_{n+1}$,
    these circles remain topologically distinct after filling in all of
    $E-\{0\}$. This finishes the proof in this case.

    Now suppose $\Sigma$ is of the second type. We first claim that, by
    forgetting all but the maximal end of $E(\Sigma)$ and the class $E(x)$
    of its immediate predecessor, we get a map $(\Sigma,\ast) \to
    (\Sigma_F,\ast)$. Taking the same approach as above, realize
    $E(\Sigma)$ as a closed subset $E \subset K$ with the maximal end at
    the origin. By Richards, the surface $(S^2\setminus E,\infty)$ is
    homeomorphic to $(\Sigma,\ast)$. We claim the origin is the only
    accumulation point of $E(x)$. Since $E(x)$ has no successor except the
    origin, any other accumulation point of $E(x)$ must be equivalent to
    $x$. But then every point in $E(x)$ is an accumulation point of $E(x)$.
    This makes $E(x) \cup \{0\}$ a closed and perfect subset of $K$, so it
    is homeomorphic to $K$, contradicting our assumption that $E(x)$ has
    countable cardinality. Thus, for any compact interval $I \subset
    (0,1]$, $I\cap E(x)$ has finite cardinality, so we can enumerate $E(x)$
    as a decreasing sequence $\{x_n\}_{n \in\N} \subset E$ converging to
    $0$. This shows $\left( S^2 \setminus (E(x)\cup \{0\} ), \infty \right)
    \cong (\Sigma_F,\ast)$. Since mapping classes induce homeomorphisms
	of the ends space which, as noted in Section \ref{sec:preliminaries}, 
	preserve the equivalence class of each end, 
	we obtain a well-defined map $\mcg(\Sigma,
    \ast) \to \mcg(\Sigma_F, \ast)$. To finish, take any point $y_n \in
    [x_{n+1},x_n]$ such that $\set{y_n} \cap E = \emptyset$.  Then the
    circles $\{\alpha_n\}_{n \in \N}$ of radius $y_n$ centered at the
    origin are pairwise disjoint and we can extract from them a subsequence
    that project to a good sequence of curves on $\Sigma_F$. This finishes
    the proof. \qedhere

  \end{proof}
  
  \begin{remark}

    Note that for $\Sigma$ of the second type in Theorem \ref{thm:bigabel},
    we do not need the maximal end $y$ to have a unique immediate
    predecessor. This is because the mapping class group always preserves
    equivalence classes of ends, so even if $y$ has other immediate
    predecessors, the map forgetting all ends except $y$ and $E(x)$ still
    induces a well-defined homomorphism on the level of mapping class
    groups. 
    
  \end{remark}

  \begin{remark}
    
    In our setting above, it seems plausible that the forgetful map from
    $\mcg(\Sigma)$ to either $\mcg(\Sigma_L)$ or $\mcg(\Sigma_F)$ is
    surjective, but we will not pursue that statement here.

  \end{remark}

  We record some consequences of Theorem \ref{thm:bigabel}.

  \begin{corollary} \label{cor:bigabel}

    Suppose $\Sigma$ is a surface that satisfies one of the descriptions in
    Theorem \ref{thm:bigabel}. Let $\ast \in \Sigma$ be a marked point. Let
    $G$ be either $\Homeo^+(\Sigma)$, $\mcg(\Sigma)$,
    $\Homeo^+(\Sigma,\ast)$, or $\mcg(\Sigma,\ast)$. Then $G$ is not
    perfect, is not generated by torsion elements, and does not have the
    automatic continuity property.

  \end{corollary}
  
  \begin{proof}
    
    Since a Polish group is separable, it can have at most
    $\mathfrak{c}=2^{\aleph_0}$ continuous epimorphisms to $\Q$. But
    $\bigoplus_{2^{\aleph_0}} \Q$ has $2^{\mathfrak{c}}$ epimorphisms to
    $\Q$. So, by Theorem \ref{thm:bigabel}, $\mcg(\Sigma)$ is not perfect,
    is not generated by torsion elements, and does not have the automatic
    continuity property. These three properties are inherited by quotients,
    so $\Homeo^+(\Sigma)$ also cannot have any of these properties. The
    same argument applies to a marked $\Sigma$. \qedhere 

  \end{proof}

  \begin{remark}

    If $E$ is a countable ends space homeomorphic to $\omega^\alpha + 1$,
    for some countable successor ordinal $\alpha$, then $\Sigma =S^2
    \setminus E$ is a surface of type $2$ of Theorem \ref{thm:bigabel}, by
    \cite[Proposition 4.3]{MR21}. This gives Theorem \ref{introthm4} of the
    introduction.

  \end{remark}

  \subsection{Topological generation by involutions}

  \begin{theorem} \label{thm:topgen}

    Let $\Sigma$ be either $\R^2 \setminus \N$ or the infinite genus
    surface with one end. Then $\mcg(\Sigma)$ is topologically generated by
    involutions and is the topological closure of the normal closure of a
    single involution. Consequentially, $\HH^1(\mcg(\Sigma),\Z)=0$.

  \end{theorem}

  \begin{proof}
    
    We first focus on $\Sigma = \R^2 \setminus \N$. The beginning of the
    proof is very similar to that of Theorem \ref{thm:unmarkedhomeo}. Note
    that $\R^2 \setminus \N$ is homeomorphic to $\R^2 \setminus \Z^2$. This
    is because both surfaces have genus $0$, and their ends spaces are
    homeomorphic. The advantage of viewing the surface as $\R^2 \setminus
    \Z^2$ is as follows. 
    
    Let $\tau$ be the rotation in the plane by angle $\pi$ centered at the
    origin, i.e. $\tau(x+iy) = e^{i\pi}(x+iy)$. We also have the
    translation $\phi(x+iy) = (x+1)+iy$. Both maps preserve $\Z^2$, so they
    induce homeomorphisms of $\Sigma$, where $\tau$ has order $2$. One
    checks that $[\phi,\tau] = \phi \tau \phi^{-1} \tau = \phi^2.$

    We define a \emph{half-space} of $\Sigma$ to be a closed subset $H
    \subset \Sigma$, such that $\partial H$ is a properly embedded simple
    arc joining infinity to itself, and both $H$ and $H^c$ contain
    infinitely many punctures (isolated ends) of $\Sigma$. We will consider
    an explicit half-space in $\Sigma$. Let $h(x) = \sec(\pi x)-.5$ with
    domain $(-.5,.5)$. The graph of $h(x)$ is a convex curve that misses
    all of $\Z^2$ and is contained in the vertical strip $\set{(x,y): -.5 \le
    x \le .5}$. See figure \ref{fig:flute}. The set $H= \{(x,y) \in \Sigma:
    y \ge h(x)\}$ is a half-space, and $\phi^2$ is an $H$--free translation,
    in the sense that $\{ \phi^{2n} H\}_{n \in \Z}$ are pairwise disjoint.
    Therefore, with the same swindle as before, we obtain 
    \[\Homeo(H,\partial H) \le \langle \langle \phi^2 \rangle \rangle \le
    \langle \langle \tau \rangle \rangle.\] 
    While the swindle still works, the rest of the proof for Theorem
    \ref{thm:unmarkedhomeo} does (and should) not work. The only statement
    that seems to fail is Lemma \ref{lemma:halfspaceintersections}, (and so
    Theorem \ref{thm:halfspacehomeogen} also fails in this case).
    
    \begin{figure}[htp!] 
      \begin{center}
       \includegraphics{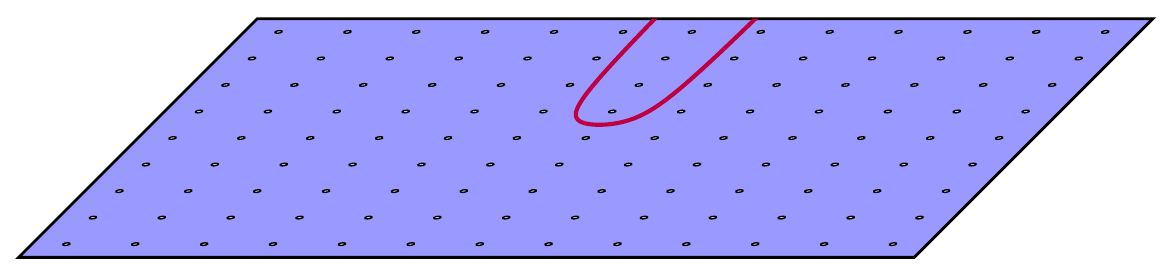} 
       \end{center} 
       \caption{The half-space $H$ in $\R^2 \setminus \Z^2$.}
      \label{fig:flute} 
    \end{figure}

    We now move to the mapping class group. As before, to simplify the
    discussion, we will conflate half-spaces and simple closed curves with
    their ambient isotopy classes. We will keep on denoting $\phi$ and
    $\tau$ for their mapping classes. 
    
    Consider the short exact sequence $$ 1 \to \pmcg(\Sigma) \to
    \mcg(\Sigma) \to \Homeo(E(\Sigma)) \to 1,$$ where $\pmcg(\Sigma)$ is
    called the \emph{pure mapping class group}, i.e.\ the subgroup fixing
    each end of $\Sigma$. Since $\Sigma$ has no genus, by \cite{PV18},
    $\pmcg(\Sigma) = \overline{\pmcg_c(\Sigma)}$, where $\pmcg_c(\Sigma)$
    is the subgroup of compactly supported mapping classes. Since Dehn
    twists generate the pure mapping class group of any compact surface,
    $\pmcg(\Sigma)$ is topologically generated by Dehn twists. The goal now
    is to show every Dehn twist in $\mcg(\Sigma)$ is contained in the
    normal closure of $\mcg(H, \partial H)$, and that the normal closure of
    $\mcg(H,\partial H)$ surjects onto $\Homeo(E(\Sigma))$ 
    
    We first deal with the Dehn twists. Let $\alpha \subset \Sigma$ be any
    simple closed curve. Then $\alpha$ bounds a topological disk containing
    finitely many points of $\Z^2$. Choose a simple closed curve $\beta
    \subset H$ that bounds an equal number of points of $\Z^2$. We can find
    a homeomorphism $f \in \Homeo^+(\Sigma)$, such that $f(\alpha)=\beta$.
    This is simply the change-of-coordinate principle made possible by the
    classification of surfaces. We now have  
    $$T_\alpha = T_{f^{-1}(\beta)} = f^{-1} T_\beta f \in \langle \langle 
    \mcg(H,\partial H) \rangle \rangle.$$

    To show $\langle \langle \mcg(H,\partial H) \rangle \rangle$ surjects
    onto $\Homeo(E(\Sigma))$, we produce sufficiently many permutations of
    non-maximal ends. First note that $E(\Sigma)$ has exactly one maximal
    end, represented by $\infty$, which must be invariant under any
    homeomorphism. Every other end is isolated, so $\Homeo(E(\Sigma))$ is
    nothing other than the permutation group $\Sym(\Z^2)$ on $\Z^2$. Within
    $H$, pair off infinitely many punctures/ends $\{(x_{i, 1}, x_{i,
    2})\}_{i \in I}$ such that $x_{i, 2}$ is directly above $x_{i, 1}$
    and all the pairs are pairwise disjoint.  It is clear that $\mcg(H,
    \partial H)$ contains a mapping class $f$ which transposes all pairs
    simultaneously. Note that $\bigcup_{i \in I}\{(x_{i, 1}, x_{i, 2})\}$ is both infinite and co-infinite in
    $E(\Sigma)$. Since by \cite{Ric63}, $\mcg(\Sigma)$ surjects onto
    $\Homeo(E(\Sigma)) = \Sym(\Z^2)$, the image of $\langle \langle
    \mcg(H,\partial H) \rangle \rangle$ in $\Sym(\Z^2)$ contains all order
    $2$ permutations supported on infinite, co-infinite subsets. It is
    straightforward to show this set generates $\Sym(\Z^2)$. In summary, we
    have shown $\langle \langle \mcg(H,\partial H) \rangle \rangle$
    topologically generates $\pmcg(\Sigma)$ and surjects onto
    $\Homeo(E(\Sigma))$. This yields
    \[ \mcg(\Sigma) = \overline{\langle \langle \mcg(H,\partial H) \rangle
    \rangle} = \overline{\langle \langle \phi^2 \rangle \rangle} =
    \overline{\langle \langle \tau \rangle \rangle}.\]
    
    To go from $\R^2 \setminus \Z^2$ to the one-ended infinite genus
    surface $\Sigma_L$ we observe that instead of removing the integer
    lattice points from $\R^2$, we can remove a small disk from each
    lattice point and glue on a handle to get a surface $\Sigma'$
    homeomorphic to $\Sigma_L$. Furthermore, we can make sure $\tau$ and
    $\phi$ preserve $\Sigma'$. A half-space in $\Sigma'$ is simply a closed
    component of a dividing arc that cuts off two component of infinite
    genus. The explicit half-space $H$ we defined for $\R^2 \setminus \Z^2$
    can also be made into a half-space here. Then running the same argument
    as above and observing that $\pmcg(\Sigma') = \mcg(\Sigma')$ completes
    the proof. 
    
    The last statement about the cohomology of these groups follows from
    the fact that any homomorphism from a Polish group to $\Z$ is
    automatically continuous \cite{Dud61}.\qedhere

  \end{proof}

 
  \bibliographystyle{abbrv} 
  \bibliography{biginvolutions.bib}


  \bigskip
  \bigskip
  \bigskip

  \begin{center}
    \begin{tabular}{|p{2.1in}@{\qquad\qquad\qquad}|p{2.1in}}
      Justin Malestein 
      \newline
      Department of Mathematics
      \newline
      University of Oklahoma
      \newline
      \texttt{justin.malestein@ou.edu}
      &
      Jing Tao 
      \newline
      Department of Mathematics
      \newline
      University of Oklahoma
      \newline
      \texttt{jing@ou.edu}
    \end{tabular}
  \end{center}

  \end{document}